\documentclass{amsart}
\usepackage{amsmath,amsthm,amsfonts,amssymb}

\numberwithin{equation}{section}
\theoremstyle{plain}

\newtheorem{theorem}{Theorem}
\newtheorem{assumption}{Assumption}

\newtheorem{lemma}{Lemma}

\newtheorem{remark}{Remark}

\begin{document}

\title[CLT for Nonlinear Hawkes Processes]{Central Limit Theorem for Nonlinear Hawkes Processes}
\author{LINGJIONG ZHU}
\address
{Courant Institute of Mathematical Sciences\newline
\indent New York University\newline
\indent 251 Mercer Street\newline
\indent New York, NY-10012\newline
\indent United States of America}
\email{ling@cims.nyu.edu}
\date{4 April 2012. \textit{Revised:} 23 November 2012}
%point processes, central limit and other weak theorems.

\thanks{This research was supported partially by a grant from the National Science Foundation: DMS-0904701, DARPA grant and MacCracken Fellowship 
from New York University.}

\begin{abstract}
Hawkes process is a self-exciting point process with clustering effect whose intensity depends on its entire past history. It has wide applications
in neuroscience, finance and many other fields. In this paper, we obtain a functional central limit theorem for nonlinear Hawkes process.
Under the same assumptions, we also obtain a Strassen's invariance principle, i.e. a functional law of the iterated logarithm.
\end{abstract}

\maketitle

\section{Introduction and Main Results}

\subsection{Introduction}

Hawkes process is a self-exciting simple point process first introduced by Hawkes \cite{Hawkes}. The future
evolution of a self-exciting point process is influenced by the timing of past events. The process is non-Markovian except
for some very special cases. In other words, Hawkes process depends on the entire past history and has a long memory. 
Hawkes process has wide applications in neuroscience, seismology, genome analysis, finance and many other fields. 
It has both self-exciting and clustering properties, which is very appealing to some financial applications. 
According to Errais et al. \cite{Errais}, ``The collapse of Lehman Brothers brought the financial system to the brink of a breakdown.
The dramatic repercussions point to the exisence of feedback phenomena that are channeled through the complex web of informational and contractual
relationships in the economy... This and related episodes motivate the design of models of correlated default timing that incorporate
the feedback phenomena that plague credit markets.'' The self-exciting and clustering properties of Hawkes process make it a viable candidate in modeling
the correlated defaults and evaluating the credit derivatives in finance, for example, see 
Errais et al. \cite{Errais} and Dassios and Zhao \cite{Dassios}. 

Most of the literature of Hawkes processes studies only the linear case, which has an immigration-birth representation (see Hawkes and Oakes \cite{HawkesII}).
The stability, law of large numbers, central limit theorem, large deviations, Bartlett spectrum etc. have all been studied and 
well understood.
Almost all of the applications of Hawkes process in the literatures consider exclusively the linear case. 
Because of the lack of immigration-birth representation and computational tractability, nonlinear Hawkes process is much less studied. 
However, some efforts have already been made in this direction. For instance, see Br\'{e}maud and Massouli\'{e} \cite{Bremaud}, 
Zhu \cite{ZhuI} and Zhu \cite{ZhuII}. 
In this paper, we will prove a functional central limit theorem for nonlinear Hawkes process. Hopefully, in the future, nonlinear Hawkes processes
will also be used in the applications in various fields.

For a list of references on the theories and applications of Hawkes process, we refer to Daley and Vere-Jones \cite{Daley} and Liniger \cite{Liniger}.

\subsection{Nonlinear Hawkes Processes}

Let $N$ be a simple point process on $\mathbb{R}$ and let $\mathcal{F}^{-\infty}_{t}:=\sigma(N(C),C\in\mathcal{B}(\mathbb{R}), C\subset(-\infty,t])$ be
an increasing family of $\sigma$-algebras. Any nonnegative $\mathcal{F}^{-\infty}_{t}$-progressively measurable process $\lambda_{t}$ with
\begin{equation}
\mathbb{E}\left[N(a,b]|\mathcal{F}^{-\infty}_{a}\right]=\mathbb{E}\left[\int_{a}^{b}\lambda_{s}ds\big|\mathcal{F}^{-\infty}_{a}\right]
\end{equation}
a.s. for all intervals $(a,b]$ is called an $\mathcal{F}^{-\infty}_{t}$-intensity of $N$. We use the notation $N_{t}:=N(0,t]$ to denote the number of
points in the interval $(0,t]$. 

A general Hawkes process is a simple point process $N$ admitting an $\mathcal{F}^{-\infty}_{t}$-intensity
\begin{equation}
\lambda_{t}:=\lambda\left(\int_{-\infty}^{t}h(t-s)N(ds)\right),\label{dynamics}
\end{equation}
where $\lambda(\cdot):\mathbb{R}^{+}\rightarrow\mathbb{R}^{+}$ is locally integrable, left continuous, 
$h(\cdot):\mathbb{R}^{+}\rightarrow\mathbb{R}^{+}$ and
we always assume that $\Vert h\Vert_{L^{1}}=\int_{0}^{\infty}h(t)dt<\infty$. 
In \eqref{dynamics}, $\int_{-\infty}^{t}h(t-s)N(ds)$ stands for $\int_{(-\infty,t)}h(t-s)N(ds)=\sum_{\tau<t}h(t-\tau)$, where
$\tau$ are the occurences of the points before time $t$.

In the literature, $h(\cdot)$ and $\lambda(\cdot)$ are usually referred to
as exciting function and rate function respectively.

A Hawkes process is linear if $\lambda(\cdot)$ is linear and it is nonlinear otherwise.

Br\'{e}maud and Massouli\'{e} \cite{Bremaud} proved that under the assumption
that $\lambda(\cdot)$ is $\alpha$-Lipschitz with $\alpha\Vert h\Vert_{L^{1}}<1$, 
there exists a unique stationary and ergodic version of Hawkes process satisfying the dynamics \eqref{dynamics}.

Br\'{e}maud and Massouli\'{e} \cite{Bremaud} studied the stability of nonlinear Hawkes process in great details, including existence, uniqueness,
stability in distribution and in variation etc.

Later, Br\'{e}maud et al. \cite{BremaudII} studied the rate of convergence of nonlinar Hawkes process to its stationary version.

\subsection{Limit Theorems for Hawkes Processes}

When $\lambda(\cdot)$ is linear, say $\lambda(z)=\nu+z$, for some $\nu>0$ and $\Vert h\Vert_{L^{1}}<1$, 
Hawkes process has a very nice immigration-birth representation, see
for example Hawkes and Oakes \cite{HawkesII}. For the linear Hawkes process, limit theorems are very well understood. There is the law of large numbers 
(see for instance Daley and Vere-Jones \cite{Daley}), i.e.
\begin{equation}
\frac{N_{t}}{t}\rightarrow\frac{\nu}{1-\Vert h\Vert_{L^{1}}},\quad\text{as $t\rightarrow\infty$ a.s.}
\end{equation}
Moreover, Bordenave and Torrisi \cite{Bordenave} proved a large deviation principle for $(\frac{N_{t}}{t}\in\cdot)$ with the rate function
\begin{equation}
I(x)=
\begin{cases}
x\log\left(\frac{x}{\nu+x\Vert h\Vert_{L^{1}}}\right)-x+x\Vert h\Vert_{L^{1}}+\nu &\text{if $x\in[0,\infty)$}
\\
+\infty &\text{otherwise}
\end{cases}.
\end{equation}
Recently, Bacry et al. \cite{Bacry} proved a functional central limit theorem for the linear multivariate Hawkes process under certain assumptions.
That includes the linear Hawkes process as a special case and they proved that
\begin{equation}
\frac{N_{\cdot t}-\cdot\mu t}{\sqrt{t}}\rightarrow\sigma B(\cdot),\quad\text{as $t\rightarrow\infty$,}
\end{equation}
where $B(\cdot)$ is a standard Brownian motion. The convergence
is weak convergence on $D[0,1]$, the space of c\'{a}dl\'{a}g functions on $[0,1]$, equipped with Skorokhod topology.
Here, 
\begin{equation}
\mu=\frac{\nu}{1-\Vert h\Vert_{L^{1}}}\quad\text{and}\quad\sigma^{2}=\frac{\nu}{(1-\Vert h\Vert_{L^{1}})^{3}}.
\end{equation}
Very recently, Karabash and Zhu \cite{Karabash} obtained central limit theorem and large deviation principle
for the linear Hawkes process with random marks.
In a nutshell, the linear Hawkes process satisfies very nice limit theroems and the limits can be computed more or less explicitly.

On the contrary, when $\lambda(\cdot)$ is nonlinear, the usual immigration-birth representation no longer works and you may have to use some
abstract theory to obtain limit theorems. Some progress has already been made for nonlinear Hawkes process. 

Br\'{e}maud and Massouli\'{e} \cite{Bremaud}'s stability result implies that by the erogdic theorem,
\begin{equation}
\frac{N_{t}}{t}\rightarrow\mu:=\mathbb{E}[N[0,1]],
\end{equation}
as $t\rightarrow\infty$, where $\mathbb{E}[N[0,1]]$ is the mean of $N[0,1]$ under the stationary and ergodic measure. 

When $h(\cdot)$ is exponential (and $\lambda(\cdot)$ is nonlinear), the Hawkes process is Markovian 
and Zhu \cite{ZhuI} obtained a large deviation principle for $(N_{t}/t\in\cdot)$
in this case. Zhu \cite{ZhuI} also proved the large deviation principle for the case when $h(\cdot)$ is a sum of exponentials and used that as an approximation
to recover the result for the linear case proved in Bordenave and Torrisi \cite{Bordenave}.

For the most general $h(\cdot)$ and $\lambda(\cdot)$, Zhu \cite{ZhuII} proved a process-level, i.e. level-3 large deviation principle for the Hawkes process
and used contraction principle to obtain a large deviation principle for $(N_{t}/t\in\cdot)$.

In this paper, we will prove a functional central limit theorem 
and a functional law of the iterated logarithm for nonlinear Hawkes process.

\subsection{Main Results}

The following is the assumption we will use throughout this paper.

\begin{assumption}\label{Assumption}
We assume that
\begin{itemize}
\item
$h(\cdot):[0,\infty)\rightarrow\mathbb{R}^{+}$ is a decreasing function and $\int_{0}^{\infty}th(t)dt<\infty$.

\item
$\lambda(\cdot)$ is positive and increasing and $\alpha$-Lipschitz (i.e. $|\lambda(x)-\lambda(y)|\leq\alpha|x-y|$ for any $x,y$) 
such that $\alpha\Vert h\Vert_{L^{1}}<1$.
\end{itemize}
\end{assumption}
Br\'{e}maud and Massouli\'{e} \cite{Bremaud} proved that if $\lambda(\cdot)$ is $\alpha$-Lipschitz
with $\alpha\Vert h\Vert_{L^{1}}<1$, 
there exists a unique stationary and ergodic Hawkes process satisfying the dynamics \eqref{dynamics}.
Hence, under our Assumption \ref{Assumption} (which is slightly stronger than \cite{Bremaud}), there exists a unique
stationary and ergodic Hawkes process satisfying the dynamics \eqref{dynamics}.

Let $\mathbb{P}$ and $\mathbb{E}$ denote the probability measure and expectation for a stationary, ergodic Hawkes process, 
and let $\mathbb{P}(\cdot|\mathcal{F}^{-\infty}_{0})$
and $\mathbb{E}(\cdot|\mathcal{F}^{-\infty}_{0})$ denote the conditional probability measure and conditional expectation 
for the Hawkes process given the past history.

The following are the main results of this paper.

\begin{theorem}\label{mainthm}
Under Assumption \ref{Assumption}, let $N$ be the stationary and ergodic nonlinear Hawkes process with dynamics \eqref{dynamics}. We have
\begin{equation}
\frac{N_{\cdot t}-\cdot\mu t}{\sqrt{t}}\rightarrow\sigma B(\cdot),\quad\text{as $t\rightarrow\infty$,}\label{convergence}
\end{equation}
where $B(\cdot)$ is a standard Brownian motion and $0<\sigma<\infty$, where
\begin{equation}
\sigma^{2}:=\mathbb{E}[(N[0,1]-\mu)^{2}]+2\sum_{j=1}^{\infty}\mathbb{E}[(N[0,1]-\mu)(N[j,j+1]-\mu)].\label{sigmasquare}
\end{equation}
The convergence in \eqref{convergence}
is weak convergence on $D[0,1]$, the space of c\'{a}dl\'{a}g functions on $[0,1]$, equipped with Skorokhod topology.
\end{theorem}

\begin{remark}
By a standard central limit theorem for martingales, i.e. Theoerem \ref{BTheoremII}, it is easy to see that
\begin{equation}
\frac{N_{\cdot t}-\int_{0}^{\cdot t}\lambda_{s}ds}{\sqrt{t}}\rightarrow\sqrt{\mu} B(\cdot),\quad\text{as $t\rightarrow\infty$,}
\end{equation}
where $\mu=\mathbb{E}[N[0,1]]$. In the linear case, say $\lambda(z)=\nu+z$, Bacry et al. \cite{Bacry} proved that $\sigma^{2}$ in 
\eqref{sigmasquare} satisfies $\sigma^{2}=\frac{\nu}{(1-\Vert h\Vert_{L^{1}})^{3}}>\mu=\frac{\nu}{1-\Vert h\Vert_{L^{1}}}$. That is
not surprising because $N_{\cdot t}-\cdot\mu t$ ``should'' have more fluctuations than $N_{\cdot t}-\int_{0}^{\cdot t}\lambda_{s}ds$.
Therefore, we guess that for nonlinear $\lambda(\cdot)$, $\sigma^{2}$ defined in \eqref{sigmasquare} should also satisfy $\sigma^{2}>\mu=\mathbb{E}[N[0,1]]$.
However, it might not be very easy to compute and say something about $\sigma^{2}$ in such a case.
\end{remark}

In the classical case for a sequence of i.i.d. random variables $X_{i}$ with mean $0$ and variance $1$, we have
the central limit theorem $\frac{1}{\sqrt{n}}\sum_{i=1}^{n}X_{i}\rightarrow N(0,1)$ as $n\rightarrow\infty$,
and we also have $\frac{\sum_{i=1}^{n}X_{i}}{\sqrt{n\log\log n}}\rightarrow 0$
in probability as $n\rightarrow\infty$, but the convergence does not hold a.s. The law of the iterated logarithm says
that $\limsup_{n\rightarrow\infty}\frac{\sum_{i=1}^{n}X_{i}}{\sqrt{n\log\log n}}=\sqrt{2}$ a.s. A functional version
of the law of the iterated logarithm is called Strassen's invariance principle.

It turns out that we also have a Strassen's invariance principle for nonlinear Hawkes processes under Assumption \ref{Assumption}.
\begin{theorem}\label{LIL}
Under Assumption \ref{Assumption}, let $N$ be the stationary and ergodic nonlinear Hawkes process with dynamics \eqref{dynamics}.
Let $X_{n}:=N[n-1,n]-\mu$, $S_{n}:=\sum_{i=1}^{n}X_{i}$, $s_{n}^{2}:=\mathbb{E}[S_{n}^{2}]$, $g(t)=\sup\{n: s_{n}^{2}\leq t\}$, and
for $t\in[0,1]$, let $\eta_{n}(t)$ be the usual linear interpolation, i.e.
\begin{equation}
\eta_{n}(t)=\frac{S_{k}+(s_{n}^{2}t-s_{k}^{2})(s_{k+1}^{2}-s_{k}^{2})^{-1}X_{k+1}}
{\sqrt{2s_{n}^{2}\log\log s_{n}^{2}}},\quad s_{k}^{2}\leq s_{n}^{2}t\leq s_{k+1}^{2},k=0,1,\ldots,n-1.
\end{equation}
Then, $g(e)<\infty$, $\{\eta_{n},n>g(e)\}$ is relatively compact in $C[0,1]$, the set of continuous functions
on $[0,1]$ equipped with uniform topology, and the set of limit points is
the set of absolutely continuous functions $f(\cdot)$ on $[0,1]$ such that $f(0)=0$ and $\int_{0}^{1}f'(t)^{2}dt\leq 1$.
\end{theorem}

\section{Proofs}

This section is devoted to the proof of Theorem \ref{mainthm}. We use a standard central limit theorem, i.e. Theorem \ref{BTheorem}. 
In our proof, we need the fact that $\mathbb{E}[N[0,1]^{2}]<\infty$, which is proved in Lemma \ref{secondmoment}.
Lemma \ref{secondmoment} is proved by proving a stronger result first, i.e. Lemma \ref{midstep}.
We will also prove Lemma \ref{positivesigma} to guarantee that $\sigma>0$ so that the central limit theorem is not degenerate.

Let us first quote the two necessary central limit theorems from Billingsley \cite{Billingsley}. In both Theorem \ref{BTheorem}
and Theorem \ref{BTheoremII}, the filtrations are natural the ones, i.e. given a stochastic process $(X_{n})_{n\in\mathbb{Z}}$,
$\mathcal{F}^{a}_{b}:=\sigma(X_{n},a\leq n\leq b)$, for $-\infty\leq a\leq b\leq\infty$.

\begin{theorem}[Page 197 \cite{Billingsley}]\label{BTheorem}
Suppose $X_{n}$, $n\in\mathbb{Z}$, is an ergodic stationary sequence such that $\mathbb{E}[X_{n}]=0$ and
\begin{equation}
\sum_{n\geq 1}\Vert\mathbb{E}[X_{0}|\mathcal{F}^{-\infty}_{-n}]\Vert_{2}<\infty,
\end{equation}
where $\Vert Y\Vert_{2}=(\mathbb{E}[Y^{2}])^{1/2}$. Let $S_{n}=X_{1}+\cdots+X_{n}$. 
Then $S_{[n\cdot]}/\sqrt{n}\rightarrow\sigma B(\cdot)$ weakly, where the weak convergence is on $D[0,1]$ equipped with the Skorohod topology and
$\sigma^{2}=\mathbb{E}[X_{0}^{2}]+2\sum_{n=1}^{\infty}\mathbb{E}[X_{0}X_{n}]$. The series converges absolutely.
\end{theorem}

\begin{theorem}[Page 196 \cite{Billingsley}]\label{BTheoremII}
Suppose $X_{n}$, $n\in\mathbb{Z}$, is an erogdic, stationary sequence of square integrable martingale differences,
i.e. $\sigma^{2}=\mathbb{E}[X_{n}^{2}]<\infty$, and let $\mathbb{E}[X_{n}|\mathcal{F}^{-\infty}_{n-1}]=0$. Let
$S_{n}=X_{1}+\cdots+X_{n}$. Then $S_{[n\cdot]}/\sqrt{n}\rightarrow\sigma B(\cdot)$ weakly, where
the weak convergence is on $D[0,1]$ equipped with the Skorohod topology.
\end{theorem}

Now, we are ready to prove our main result.

\begin{proof}[Proof of Theorem \ref{mainthm}]
Since in the stationary regime, $\mathbb{E}[N[n,n+1]]=\mathbb{E}[N[0,1]]$ for any $n\in\mathbb{Z}$ and let us denote
$\mathbb{E}[N[0,1]]=\mu$. In order to apply Theorem \ref{BTheorem}, let us first prove that
\begin{equation}\label{finitesum}
\sum_{n=1}^{\infty}\left\{\mathbb{E}\left[\left(\mathbb{E}[N(n,n+1]-\mu|\mathcal{F}^{-\infty}_{0}]\right)^{2}\right]\right\}^{1/2}<\infty.
\end{equation}
Let $\mathbb{E}^{\omega^{-}_{1}}[N(n,n+1]]$ and $\mathbb{E}^{\omega^{-}_{2}}[N(n,n+1]]$ be two independent copies of
$\mathbb{E}[N(n,n+1]|\mathcal{F}^{-\infty}_{0}]$. It is easy to check that
\begin{align}
&\frac{1}{2}\mathbb{E}\left\{\left[\mathbb{E}^{\omega^{-}_{1}}[N(n,n+1]]-\mathbb{E}^{\omega^{-}_{2}}[N(n,n+1]]\right]^{2}\right\}
\\
&=\frac{1}{2}\mathbb{E}\left[\mathbb{E}^{\omega^{-}_{1}}[N(n,n+1]]^{2}\right]
+\frac{1}{2}\mathbb{E}\left[\mathbb{E}^{\omega^{-}_{2}}[N(n,n+1]]^{2}\right]\nonumber
\\
&\phantom{=\frac{1}{2}\mathbb{E}\left[\mathbb{E}^{\omega^{-}_{1}}[N(n,n+1]]^{2}\right]}
-\mathbb{E}\left[\mathbb{E}^{\omega^{-}_{1}}[N(n,n+1]]\mathbb{E}^{\omega^{-}_{2}}[N(n,n+1]]\right]\nonumber
\\
&=\mathbb{E}\left[\mathbb{E}[N(n,n+1]|\mathcal{F}^{-\infty}_{0}]^{2}\right]-\mu^{2}\nonumber
\\
&=\mathbb{E}\left[(\mathbb{E}[N(n,n+1]-\mu|\mathcal{F}^{-\infty}_{0}])^{2}\right].\nonumber
\end{align}
Therefore, we have
\begin{align}
&\mathbb{E}\left[(\mathbb{E}[N(n,n+1]-\mu|\mathcal{F}^{-\infty}_{0}])^{2}\right]
\\
&=\frac{1}{2}\mathbb{E}\left\{\left[\mathbb{E}^{\omega^{-}_{1}}[N(n,n+1]]-\mathbb{E}^{\omega^{-}_{2}}[N(n,n+1]]\right]^{2}\right\}\nonumber
\\
&\leq\mathbb{E}\left\{\left[\mathbb{E}^{\omega^{-}_{1}}[N(n,n+1]]-\mathbb{E}^{\emptyset}[N(n,n+1]]\right]^{2}\right\}\nonumber
\\
&\phantom{\leq\mathbb{E}\mathbb{E}^{\omega^{-}_{1}}[N(n,n+1]]}
+\mathbb{E}\left\{\left[\mathbb{E}^{\omega^{-}_{2}}[N(n,n+1]]-\mathbb{E}^{\emptyset}[N(n,n+1]]\right]^{2}\right\}\nonumber
\\
&=2\mathbb{E}\left\{\left[\mathbb{E}^{\omega^{-}_{1}}[N(n,n+1]]-\mathbb{E}^{\emptyset}[N(n,n+1]]\right]^{2}\right\},\nonumber
\end{align}
where $\mathbb{E}^{\emptyset}[N(n,n+1]]$ denotes the expectation of the number of points in $(n,n+1]$ for the 
Hawkes process with the same dynamics \eqref{dynamics} and empty history, i.e. $N(-\infty,0]=0$.

Next, let us estimate $\mathbb{E}^{\omega^{-}_{1}}[N(n,n+1]]-\mathbb{E}^{\emptyset}[N(n,n+1]]$. 
$\mathbb{E}^{\omega^{-}_{1}}[N(n,n+1]]$
is the expectation of the number of points in $(n,n+1]$ for the Hawkes process with intensity
$\lambda_{t}=\lambda\left(\sum_{\tau: \tau\in\omega^{-}_{1}\cup\omega[0,t)}h(t-\tau)\right)$.
It is well defined for a.e. $\omega^{-}_{1}$ under $\mathbb{P}$ because, under Assumption \ref{Assumption},
\begin{equation}
\mathbb{E}[\lambda_{t}]\leq\lambda(0)+\alpha\mathbb{E}\left[\int_{-\infty}^{t}h(t-s)N(ds)\right]
=\lambda(0)+\alpha\Vert h\Vert_{L^{1}}\mathbb{E}[N[0,1]]<\infty,
\end{equation}
which implies that $\lambda_{t}<\infty$ $\mathbb{P}$-a.s.

It is clear that $\mathbb{E}^{\omega^{-}_{1}}[N(n,n+1]]\geq\mathbb{E}^{\emptyset}[N(n,n+1]]$ almost surely, 
so we can use a coupling method to estimate the difference. 
We will follow the ideas in Br\'{e}maud and Massouli\'{e} \cite{Bremaud} using the Poisson embedding method.
Consider $(\Omega,\mathcal{F},\mathcal{P})$, the canonical space of a point process on $\mathbb{R}^{+}\times\mathbb{R}^{+}$ 
in which $\overline{N}$ is Poisson with intensity $1$ under the probability measure $\mathcal{P}$. 
Then the Hawkes process $N^{0}$ with empty past history and intensity
$\lambda^{0}_{t}$ satisfies the following.
\begin{equation}
\begin{cases}
\lambda^{0}_{t}=\lambda\left(\int_{(0,t)}h(t-s)N^{0}(ds)\right)& t\in\mathbb{R}^{+},
\\
N^{0}(C)=\int_{C}\overline{N}(dt\times[0,\lambda^{0}_{t}])& C\in\mathcal{B}(\mathbb{R}^{+}).
\end{cases}
\end{equation}
For $n\geq 1$, let us define recursively $\lambda^{n}_{t}$, $D_{n}$ and $N^{n}$ as follows.
\begin{equation}\label{canonical}
\begin{cases}
\lambda^{n}_{t}=\lambda\left(\int_{(0,t)}h(t-s)N^{n-1}(ds)+\sum_{\tau\in\omega^{-}_{1}}h(t-\tau)\right)& t\in\mathbb{R}^{+},
\\
D_{n}(C)=\int_{C}\overline{N}(dt\times[\lambda^{n-1}_{t},\lambda^{n}_{t}])& C\in\mathcal{B}(\mathbb{R}^{+}),
\\
N^{n}(C)=N^{n-1}(C)+D_{n}(C) & C\in\mathcal{B}(\mathbb{R}^{+}).
\end{cases}
\end{equation}
Following the arguments as in Br\'{e}maud and Massouli\'{e} \cite{Bremaud}, we know that each $\lambda^{n}_{t}$
is an $\mathcal{F}^{\overline{N}}_{t}$-intensity of $N^{n}$, where $\mathcal{F}^{\overline{N}}_{t}$ is the $\sigma$-algebra
generated by $\overline{N}$ up to time $t$. 
By our Assumption \ref{Assumption}, $\lambda(\cdot)$ is
increasing, and it is clear that $\lambda^{n}(t)$ and $N^{n}(C)$ increase in $n$ for all $t\in\mathbb{R}^{+}$
and $C\in\mathcal{B}(\mathbb{R}^{+})$. Thus, $D_{n}$ is well defined and also that as $n\rightarrow\infty$, the limiting
processes $\lambda_{t}$ and $N$ exist. $N$ counts the number of points
of $\overline{N}$ below the curve $t\mapsto\lambda_{t}$ and admits $\lambda_{t}$ as an $\mathcal{F}^{\overline{N}}_{t}$-intensity.
By the monotonicity properties of $\lambda^{n}_{t}$ and $N^{n}$, we have
\begin{align}
&\lambda^{n}_{t}\leq\lambda\left(\int_{(0,t)}h(t-s)N(ds)+\sum_{\tau\in\omega^{-}_{1}}h(t-\tau)\right),
\\
&\lambda_{t}\geq\lambda\left(\int_{(0,t)}h(t-s)N^{n}(ds)+\sum_{\tau\in\omega^{-}_{1}}h(t-\tau)\right).
\end{align}
Letting $n\rightarrow\infty$ (it is valid since we
assume that $\lambda(\cdot)$ is Lipschitz and thus continuous), 
we conclude that $N$, $\lambda_{t}$ satisfies the dynamics \eqref{dynamics}.
Therefore, with intensity $\lambda_{t}$, $N=N^{0}+\sum_{i=1}^{\infty}D_{i}$ is the Hawkes process with past
history $\omega^{-}_{1}$.

We can then estimate the difference by noticing that
\begin{equation}
\mathbb{E}^{\omega^{-}_{1}}[N(n,n+1]]-\mathbb{E}^{\emptyset}[N(n,n+1]]=\sum_{i=1}^{\infty}\mathbb{E}^{\mathcal{P}}[D_{i}(n,n+1]].
\end{equation}
Here $\mathbb{E}^{\mathcal{P}}$ means the expectation with respect to $\mathcal{P}$, the probability measure on 
the canonical space that we defined earlier.

We have
\begin{align}\label{estimation}
&\mathbb{E}^{\mathcal{P}}[D_{1}(n,n+1]]
\\
&=\mathbb{E}^{\mathcal{P}}\left[\int_{n}^{n+1}(\lambda^{1}(t)-\lambda^{0}(t))dt\right]\nonumber
\\
&=\mathbb{E}^{\mathcal{P}}\left[\int_{n}^{n+1}\lambda\left(\sum_{\tau<t,\tau\in N^{0}\cup\omega^{-}_{1}}h(t-\tau)\right)
-\lambda\left(\sum_{\tau<t,\tau\in N^{0}\cup\emptyset}h(t-\tau)\right)dt\right]\nonumber
\\
&\leq\alpha\int_{n}^{n+1}\sum_{\tau\in\omega_{1}^{-}}h(t-\tau)dt,\nonumber
\end{align}
where the first equality in \eqref{estimation} is due to the construction of $D_{1}$ in \eqref{canonical},
the second equality in \eqref{estimation} is due to the definitions of $\lambda^{1}$ and $\lambda^{0}$ in \eqref{canonical}
and finally the inequality in \eqref{estimation} is due to the fact that
$\lambda(\cdot)$ is $\alpha$-Lipschitz by Assumption \ref{Assumption}.
Similarly,
\begin{align}
\mathbb{E}^{\mathcal{P}}[D_{2}(n,n+1]]&\leq\mathbb{E}^{\omega_{1}^{-}}\left[\alpha\int_{n}^{n+1}\sum_{\tau\in D_{1},\tau<t}h(t-\tau)dt\right]
\\
&\leq\sum_{\tau\in\omega_{1}^{-}}\alpha^{2}\int_{n}^{n+1}\int_{0}^{t}h(t-s)h(s-\tau)dsdt.\nonumber
\end{align}
Iteratively, we have, for any $k\in\mathbb{N}$,
\begin{align}
\mathbb{E}^{\mathcal{P}}[D_{k}(n,n+1]]
\leq\sum_{\tau\in\omega_{1}^{-}}
\alpha^{k}&\int_{n}^{n+1}\int_{0}^{t_{k}}\cdots\int_{0}^{t_{2}}
h(t_{k}-t_{k-1})h(t_{k-1}-t_{k-2})\nonumber
\\
&\cdots h(t_{2}-t_{1})h(t_{1}-\tau)dt_{1}\cdots dt_{k}
=:\sum_{\tau\in\omega_{1}^{-}}K_{k}(n,\tau).\nonumber
\end{align}
Now let $K(n,\tau):=\sum_{k=1}^{\infty}K_{k}(n,\tau)$. Then,
\begin{align}
&\mathbb{E}\left\{\left[\mathbb{E}^{\omega^{-}_{1}}[N(n,n+1]]-\mathbb{E}^{\emptyset}[N(n,n+1]]\right]^{2}\right\}
\\
&\leq\mathbb{E}\left[\left(\sum_{\tau\in\omega_{1}^{-}}K(n,\tau)\right)^{2}\right]\nonumber
\\
&\leq\mathbb{E}\left[\sum_{i,j\leq 0}K(n,i)K(n,j)N[i,i+1]N[j,j+1]\right]\nonumber
\\
&=\sum_{i,j\leq 0}K(n,i)K(n,j)\mathbb{E}[N[i,i+1]N[j,j+1]]\nonumber
\\
&\leq\sum_{i,j\leq 0}K(n,i)K(n,j)\frac{1}{2}\left\{\mathbb{E}[N[i,i+1]^{2}]+\mathbb{E}[N[j,j+1]^{2}]\right\}\nonumber
\\
&=\mathbb{E}[N[0,1]^{2}]\left(\sum_{i\leq 0}K(n,i)\right)^{2}.\nonumber
\end{align}
Here, $\mathbb{E}[N[0,1]^{2}]<\infty$ by Lemma \ref{secondmoment}. Therefore, we have
\begin{align}
&\sum_{n=1}^{\infty}\left\{\mathbb{E}\left[\left(\mathbb{E}[N(n,n+1]-\mu|\mathcal{F}^{-\infty}_{0}]\right)^{2}\right]\right\}^{1/2}
\\
&\leq\sqrt{2\mathbb{E}[N[0,1]^{2}]}\sum_{n=1}^{\infty}\sum_{i=-\infty}^{0}K(n,i)\nonumber
\\
&\leq\sqrt{2\mathbb{E}[N[0,1]^{2}]}\sum_{k=1}^{\infty}
\alpha^{k}\int_{0}^{\infty}\int_{0}^{t_{k}}\cdots\int_{0}^{t_{2}}\int_{-\infty}^{0}\nonumber
\\
&h(t_{k}-t_{k-1})h(t_{k-1}-t_{k-2})\cdots h(t_{2}-t_{1})h(t_{1}-s)dsdt_{1}\cdots dt_{k}.\nonumber
\end{align}
Let $H(t):=\int_{t}^{\infty}h(s)ds$. It is easy to check that $\int_{0}^{\infty}H(t)dt=\int_{0}^{\infty}th(t)dt<\infty$ by Assumption \ref{Assumption}. 
We have
\begin{align}
&\alpha^{k}\int_{0}^{\infty}\int_{0}^{t_{k}}\cdots\int_{0}^{t_{2}}\int_{-\infty}^{0}
\\
& h(t_{k}-t_{k-1})h(t_{k-1}-t_{k-2})\cdots h(t_{2}-t_{1})h(t_{1}-s)dsdt_{1}\cdots dt_{k}\nonumber
\\
&=\alpha^{k}\int_{0}^{\infty}\int_{0}^{t_{k}}\cdots\int_{0}^{t_{2}}
h(t_{k}-t_{k-1})h(t_{k-1}-t_{k-2})\cdots h(t_{2}-t_{1})H(t_{1})dt_{1}\cdots dt_{k}\nonumber
\\
&=\alpha^{k}\int_{0}^{\infty}\cdots\int_{t_{k-2}}^{\infty}\int_{t_{k-1}}^{\infty}h(t_{k}-t_{k-1})dt_{k}h(t_{k-1}-t_{k-2})dt_{k-1}
\cdots H(t_{1})dt_{1}\nonumber
\\
&=\alpha^{k}\Vert h\Vert_{L^{1}}^{k-1}\int_{0}^{\infty}H(t_{1})dt_{1}=\alpha^{k}\Vert h\Vert_{L^{1}}^{k-1}\int_{0}^{\infty}th(t)dt.\nonumber
\end{align}
Since $\alpha\Vert h\Vert_{L^{1}}<1$, we conclude that
\begin{align}
&\sum_{n=1}^{\infty}\left\{\mathbb{E}\left[\left(\mathbb{E}[N(n,n+1]-\mu|\mathcal{F}^{-\infty}_{0}]\right)^{2}\right]\right\}^{1/2}
\\
&\leq\sum_{k=1}^{\infty}\sqrt{2\mathbb{E}[N[0,1]^{2}]}\alpha^{k}\Vert h\Vert_{L^{1}}^{k-1}\int_{0}^{\infty}th(t)dt\nonumber
\\
&=\sqrt{2\mathbb{E}[N[0,1]^{2}]}\cdot\frac{\alpha}{1-\alpha\Vert h\Vert_{L^{1}}}\cdot\int_{0}^{\infty}th(t)dt<\infty.\nonumber
\end{align}
Hence, by Theorem \ref{BTheorem}, we have
\begin{equation}
\frac{N_{[\cdot t]}-\mu[\cdot t]}{\sqrt{t}}\rightarrow\sigma B(\cdot)\quad\text{as $t\rightarrow\infty$,}
\end{equation}
where
\begin{equation}
\sigma^{2}=\mathbb{E}[(N[0,1]-\mu)^{2}]+2\sum_{j=1}^{\infty}\mathbb{E}[(N[0,1]-\mu)(N[j,j+1]-\mu)]<\infty.\label{sigmadefn}
\end{equation}
By Lemma \ref{positivesigma}, $\sigma>0$. Now, finally, for any $\epsilon>0$, for $t$ sufficiently large,
\begin{align}
&\mathbb{P}\left(\sup_{0\leq s\leq 1}\left|\frac{N_{[st]}-\mu[st]}{\sqrt{t}}
-\frac{N_{st}-\mu st}{\sqrt{t}}\right|>\epsilon\right)
\\
&=\mathbb{P}\left(\sup_{0\leq s\leq 1}\left|(N_{[st]}-N_{st})+\mu(st-[st])\right|>\epsilon\sqrt{t}\right)\nonumber
\\
&\leq\mathbb{P}\left(\sup_{0\leq s\leq 1}\left|N_{[st]}-N_{st}\right|+\mu>\epsilon\sqrt{t}\right)\nonumber
\\
&\leq\mathbb{P}\left(\max_{0\leq k\leq [t],k\in\mathbb{Z}}N[k,k+1]>\epsilon\sqrt{t}-\mu\right)\nonumber
\\
&\leq([t]+1)\mathbb{P}(N[0,1]>\epsilon\sqrt{t}-\mu)\nonumber
\\
&\leq\frac{[t]+1}{(\epsilon\sqrt{t}-\mu)^{2}}\int_{N[0,1]>\epsilon\sqrt{t}-\mu}N[0,1]^{2}d\mathbb{P}\rightarrow 0,\nonumber
\end{align}
as $t\rightarrow\infty$ by Lemma \ref{secondmoment}. Hence, we conclude that $\frac{N_{\cdot t}-\cdot\mu t}{\sqrt{t}}\rightarrow\sigma B(\cdot)$
as $t\rightarrow\infty$.
\end{proof}

The following Lemma \ref{midstep} is used to prove Lemma \ref{secondmoment}.

\begin{lemma}\label{midstep}
There exists some $\theta>0$ such that 
$\sup_{t\geq 0}\mathbb{E}^{\emptyset}\left[e^{\int_{0}^{t}\theta h(t-s)N(ds)}\right]<\infty$.
\end{lemma}

\begin{proof}
Notice first that for any bounded deterministic function $f(\cdot)$,
\begin{equation}
\exp\left\{\int_{0}^{t}f(s)N(ds)-\int_{0}^{t}(e^{f(s)}-1)\lambda(s)ds\right\}
\end{equation}
is a martingale. Therefore, using the Lipschitz assumption of $\lambda(\cdot)$, 
i.e. $\lambda(z)\leq\lambda(0)+\alpha z$ and applying H\"{o}lder's inequality, for $\frac{1}{p}+\frac{1}{q}=1$, we have
\begin{align}
&\mathbb{E}^{\emptyset}\left[e^{\int_{0}^{t}\theta h(t-s)N(ds)}\right]
\\
&=\mathbb{E}^{\emptyset}
\left[e^{\int_{0}^{t}\theta h(t-s)N(ds)-\frac{1}{p}\int_{0}^{t}(e^{p\theta h(t-s)}-1)\lambda(s)ds
+\frac{1}{p}\int_{0}^{t}(e^{p\theta h(t-s)}-1)\lambda(s)ds}\right]\nonumber
\\
&\leq\mathbb{E}^{\emptyset}\left[e^{\frac{q}{p}\int_{0}^{t}(e^{p\theta h(t-s)}-1)\lambda(s)ds}\right]^{\frac{1}{q}}\nonumber
\\
&\leq\mathbb{E}^{\emptyset}\left[e^{\frac{q}{p}\int_{0}^{t}(e^{p\theta h(t-s)}-1)(\lambda(0)+\alpha\int_{0}^{s}h(s-u)N(du))ds}
\right]^{\frac{1}{q}}\nonumber
\\
&\leq\mathbb{E}^{\emptyset}\left[e^{\int_{0}^{t}\frac{q}{p}(e^{p\theta h(t-s)}-1)\alpha\int_{0}^{s}h(s-u)N(du)ds}\right]^{\frac{1}{q}}
\cdot e^{\frac{1}{p}\int_{0}^{\infty}(e^{p\theta h(s)}-1)\lambda(0)ds}.\nonumber
\end{align}
Let $C(t)=\int_{0}^{t}\frac{q}{p}(e^{p\theta h(t-s)}-1)\alpha ds$. Then, for any $t\in[0,T]$,
\begin{align}\label{Jensen}
&\mathbb{E}^{\emptyset}\left[e^{\int_{0}^{t}\frac{q}{p}(e^{p\theta h(t-s)}-1)\alpha\int_{0}^{s}h(s-u)N(du)ds}\right]
\\
&=\mathbb{E}^{\emptyset}\left[e^{\frac{1}{C(t)}\int_{0}^{t}\frac{q}{p}(e^{p\theta h(t-s)}-1)\alpha C(t)\int_{0}^{s}h(s-u)N(du)ds}\right]
\nonumber
\\
&\leq\mathbb{E}^{\emptyset}\left[\frac{1}{C(t)}\int_{0}^{t}\frac{q}{p}(e^{p\theta h(t-s)}-1)\alpha
e^{C(t)\int_{0}^{s}h(s-u)N(du)}ds\right]\nonumber
\\
&\leq\sup_{0\leq s\leq T}\mathbb{E}^{\emptyset}\left[e^{C(\infty)\int_{0}^{s}h(s-u)N(du)}\right],\nonumber
\end{align}
where in the first inequality in \eqref{Jensen}, we used the Jensen's inequality since $x\mapsto e^{x}$ is convex
and $\frac{1}{C(t)}\int_{0}^{t}\frac{q}{p}(e^{p\theta h(t-s)}-1)\alpha ds=1$, and in the second inequality
in \eqref{Jensen}, we used the fact that $C(t)\leq C(\infty)$ 
and again $\frac{1}{C(t)}\int_{0}^{t}\frac{q}{p}(e^{p\theta h(t-s)}-1)\alpha ds=1$.
Now choose $q>1$ so small that $q\alpha\Vert h\Vert_{L^{1}}<1$. Once $p$ and $q$ are fixed, choose so $\theta>0$ small that
\begin{equation}
C(\infty)=\int_{0}^{\infty}\frac{q}{p}(e^{p\theta h(s)}-1)\alpha ds<\theta.
\end{equation}
This implies that for any $t\in[0,T]$,
\begin{equation}
\mathbb{E}^{\emptyset}\left[e^{\int_{0}^{t}\theta h(t-s)N(ds)}\right]
\leq\sup_{0\leq s\leq T}\mathbb{E}^{\emptyset}\left[e^{\theta\int_{0}^{s}h(s-u)N(du)}\right]^{\frac{1}{q}}
\cdot e^{\frac{1}{p}\int_{0}^{\infty}(e^{p\theta h(s)}-1)\lambda(0)ds}.
\end{equation}
Hence, we conclude that for any $T>0$,
\begin{equation}
\sup_{0\leq t\leq T}\mathbb{E}^{\emptyset}\left[e^{\theta\int_{0}^{t}h(t-s)N(ds)}\right]
\leq e^{\int_{0}^{\infty}(e^{p\theta h(s)}-1)\lambda(0)ds}<\infty.
\end{equation}
\end{proof}

\begin{lemma}\label{secondmoment}
There exists some $\theta>0$ such that $\mathbb{E}[e^{\theta N[0,1]}]<\infty$. Hence $\mathbb{E}[N[0,1]^{2}]<\infty$.
\end{lemma}

\begin{proof}
By Assumption \ref{Assumption}, $h(\cdot)>0$ is positive and decreasing. Thus, $\delta=\inf_{t\in[0,1]}h(t)>0$. Hence,
\begin{equation}
\mathbb{E}^{\emptyset}[e^{\theta N[t-1,t]}]\leq\mathbb{E}^{\emptyset}[e^{\frac{\theta}{\delta}\int_{0}^{t}h(t-s)N(ds)}].
\end{equation}
By Lemma \ref{midstep}, we can choose $\theta>0$ so small that
\begin{equation}
\limsup_{t\rightarrow\infty}\mathbb{E}^{\emptyset}[e^{\theta N[t-1,t]}]<\infty.
\end{equation}
Finally, $\mathbb{E}[e^{\theta N[0,1]}]\leq\liminf_{t\rightarrow\infty}\mathbb{E}^{\emptyset}[e^{\theta N[t-1,t]}]<\infty$.
\end{proof}

It is intuitively clear that $\sigma>0$. But still we need a proof.

\begin{lemma}\label{positivesigma}
$\sigma>0$, where $\sigma$ is defined in \eqref{sigmadefn}.
\end{lemma}

\begin{proof}
Let $\eta_{n}=\sum_{j=n}^{\infty}\mathbb{E}[N(j,j+1]-\mu|\mathcal{F}^{-\infty}_{n+1}]$, where $\mu=\mathbb{E}[N[0,1]]$.
$\eta_{n}$ is well defined because we proved \eqref{finitesum}. To see this, notice that
\begin{align}
\Vert\eta_{n}\Vert_{2}&=\bigg\Vert\sum_{j=n}^{\infty}\mathbb{E}[N(j,j+1]-\mu|\mathcal{F}^{-\infty}_{n+1}]\bigg\Vert_{2}
\\
&\leq\sum_{j=n}^{\infty}\Vert\mathbb{E}[N(j,j+1]-\mu|\mathcal{F}^{-\infty}_{n+1}]\Vert_{2}<\infty,\nonumber
\end{align}
by \eqref{finitesum}. Also, it is easy to check that
\begin{align}
&\mathbb{E}[\eta_{n+1}-\eta_{n}+N(n,n+1]-\mu|\mathcal{F}^{-\infty}_{n+1}]
\\
&=\mathbb{E}\left[\sum_{j=n+1}^{\infty}\mathbb{E}[N(j,j+1]-\mu|\mathcal{F}^{-\infty}_{n+2}]\bigg|\mathcal{F}^{-\infty}_{n+1}\right]\nonumber
\\
&-\mathbb{E}\left[\sum_{j=n}^{\infty}\mathbb{E}[N(j,j+1]-\mu|\mathcal{F}^{-\infty}_{n+1}]\bigg|\mathcal{F}^{-\infty}_{n+1}\right]+N(n,n+1]-\mu\nonumber
\\
&=\sum_{j=n+1}^{\infty}\mathbb{E}[N(j,j+1]-\mu|\mathcal{F}^{-\infty}_{n+1}]
-\sum_{j=n+1}^{\infty}\mathbb{E}[N(j,j+1]-\mu|\mathcal{F}^{-\infty}_{n+1}]\nonumber
\\
&-N(n,n+1]+\mu+N(n,n+1]-\mu=0.\nonumber
\end{align}
Let $Y_{n}=\eta_{n-1}-\eta_{n-2}+N(n-2,n-1]-\mu$. Then, $Y_{n}$ is an ergodic, stationary sequence such that 
$\mathbb{E}[Y_{n}|\mathcal{F}^{-\infty}_{n-1}]=0$. By \eqref{finitesum}, $\mathbb{E}[Y_{n}^{2}]<\infty$ and 
by Theorem \ref{BTheoremII}, $S'_{[n\cdot]}/\sqrt{n}\rightarrow\sigma'B(\cdot)$, where $S'_{n}=\sum_{j=1}^{n}Y_{j}$.
It is clear that $\sigma=\sigma'<\infty$ 
since for any $\epsilon>0$,
\begin{align}
&\mathbb{P}\left(\max_{1\leq k\leq [n],k\in\mathbb{Z}}\frac{1}{\sqrt{n}}\sum_{j=1}^{k}(\eta_{j-1}-\eta_{j-2})>\epsilon\right)
\\
&=\mathbb{P}\left(\max_{1\leq k\leq [n],k\in\mathbb{Z}}(\eta_{k-1}-\eta_{-1})>\epsilon\sqrt{n}\right)\nonumber
\\
&\leq\mathbb{P}\left(\left\{\max_{1\leq k\leq [n],k\in\mathbb{Z}}|\eta_{k-1}|>\frac{\epsilon\sqrt{n}}{2}\right\}
\bigcup\left\{|\eta_{-1}|>\frac{\epsilon\sqrt{n}}{2}\right\}\right)\nonumber
\\
&\leq\sum_{k=1}^{[n]}\mathbb{P}\left(|\eta_{k-1}|>\frac{\epsilon\sqrt{n}}{2}\right)
+\mathbb{P}\left(|\eta_{-1}|>\frac{\epsilon\sqrt{n}}{2}\right)\nonumber
\\
&=([n]+1)\mathbb{P}\left(|\eta_{-1}|>\frac{\epsilon\sqrt{n}}{2}\right)\nonumber
\\
&\leq\frac{4([n]+1)}{\epsilon^{2}n}\int_{|\eta_{-1}|>\frac{\epsilon\sqrt{n}}{2}}|\eta_{-1}|^{2}d\mathbb{P}\rightarrow 0,\nonumber
\end{align}
as $n\rightarrow\infty$, where we used the stationarity of $\mathbb{P}$, Chebychev's inequality and \eqref{finitesum}.

Now, it becomes clear that
\begin{align}
\sigma^{2}&=(\sigma')^{2}=\mathbb{E}[Y_{1}^{2}]
\\
&=\mathbb{E}\left(\eta_{0}-\eta_{-1}+N(-1,0]-\mu\right)^{2}\nonumber
\\
&=\mathbb{E}\left(\sum_{j=0}^{\infty}\mathbb{E}[N(j,j+1]-\mu|\mathcal{F}^{-\infty}_{1}]
-\sum_{j=0}^{\infty}\mathbb{E}[N(j,j+1]-\mu|\mathcal{F}^{-\infty}_{0}]\right)^{2}.\nonumber
\end{align}
Consider $D=\{\omega:\omega^{-}\neq\emptyset, \omega(0,1]=\emptyset\}$. Notice that
$\mathbb{P}(\omega^{-}=\emptyset)=0$. By Jensen's inequality and Assumption \ref{Assumption}, we have
\begin{align}
\mathbb{P}(D)&=\int\mathbb{P}^{\omega^{-}}(N(0,1]=0)\mathbb{P}(d\omega^{-})
\\
&=\mathbb{E}\left[e^{-\int_{0}^{1}\lambda(\sum_{\tau\in\omega^{-}}h(t-\tau))dt}\right]\nonumber
\\
&\geq\exp\left\{-\mathbb{E}\int_{0}^{1}\lambda\left(\sum_{\tau\in\omega^{-}}h(t-\tau)\right)dt\right\}\nonumber
\\
&\geq\exp\left\{-\lambda(0)-\alpha\mathbb{E}\int_{0}^{1}\sum_{\tau\in\omega^{-}}h(t-\tau)dt\right\}\nonumber
\\
&\geq\exp\left\{-\lambda(0)-\alpha\mathbb{E}[N[0,1]]\cdot\Vert h\Vert_{L^{1}}\right\}>0.\nonumber
\end{align}
It is clear that given the event $D$, 
\begin{equation}
\sum_{j=0}^{\infty}\mathbb{E}[N(j,j+1]-\mu|\mathcal{F}^{-\infty}_{1}]
<\sum_{j=0}^{\infty}\mathbb{E}[N(j,j+1]-\mu|\mathcal{F}^{-\infty}_{0}]. 
\end{equation}
Therefore,
\begin{equation}
\mathbb{P}\left(\sum_{j=0}^{\infty}\mathbb{E}[N(j,j+1]-\mu|\mathcal{F}^{-\infty}_{1}]
\neq\sum_{j=0}^{\infty}\mathbb{E}[N(j,j+1]-\mu|\mathcal{F}^{-\infty}_{0}]\right)>0,
\end{equation}
which implies that $\sigma>0$.
\end{proof}

\begin{proof}[Proof of Theorem \ref{LIL}]
By Heyde and Scott \cite{Heyde}, the Strassen's invariance principle holds if we have \eqref{finitesum} and $\sigma>0$.
\end{proof}

\section*{Acknowledgements}

The author is very grateful to his advisor Professor S. R. S. Varadhan for helpful discussions and generous suggestions.  
The author also wishes to thank an annonymous referee for very careful readings of the manuscript and helpful suggestions
that greatly improved the paper.
The author is supported by NSF grant DMS-0904701, DARPA grant and MacCracken Fellowship at NYU.

\end{document}